\documentclass[a4paper,12pt, reqno]{amsart}
\usepackage[backref]{hyperref}
\usepackage{a4wide}
\usepackage{amsthm}
 \newtheorem{theorem}{Theorem}[section]
 
 \newtheorem{lemma}[theorem]{Lemma}
 \newtheorem{corollary}[theorem]{Corollary}

 \newcommand{\mc}{\mathcal}

 \newcommand{\D}{\mc{D}}

\begin{document}

\title{Diophantine approximation on lines in $\mathbb{C}^2$ with Gaussian prime constraints}
\author{Stephan Baier}
\address{Stephan Baier, Jawaharlal Nehru University, School of Physical Sciences,
Delhi 11067, India}

\email{email\_baier@yahoo.de}


\maketitle

\begin{abstract}
We study the problem of Diophantine approximation on lines in $\mathbb{C}^2$ with numerators and denominators restricted to Gaussian primes.
\end{abstract}
\maketitle


\section{Introduction}
The problem of Diophantine approximation in Gaussian integers has received a lot of attention especially after 
D. Sullivan's famous paper \cite{Sul} in which he proved an analogue of Khintchine's theorem for Gaussian integers and, more generally,
for imaginary quadratic fields. 
Despite significant progress in the study of Diophantine approximation on manifolds and lines, 
there has not been any work which deals with the problem of approximating by Gaussian integers in this particular context, though. However, 
the problem of approximating (almost all) points on a line in $\mathbb{C}^n$ which passes through the origin using Gaussian integers can be 
handled using a variation of an argument due to Beresnevich, Bernik, Dickinson, Dodson \cite{BBDD}. 
For more details, the reader is referred to a survey \cite{Gho} by A. Ghosh on Diophantine approximation on affine subspaces 
which provides a convenient collection of results in this area, including the following one on Diophantine approximation
on lines in $\mathbb{R}^2$ with prime constraints due to Ghosh and the author of the present paper \cite{BG1}.

\begin{theorem} \label{real} Let $\varepsilon>0$ and let $c>1$ be an irrational number. Then for almost all positive $\alpha$, with respect to the Lebesgue measure, there are infinitely many triples $(p,q,r)$ with $p$ and $r$ prime and $q$ an integer such that
\begin{equation*} 
0 < p\alpha - r \le p^{-1/5+\varepsilon} \quad \mbox{and} \quad 0< pc\alpha - q \le p^{-1/5+\varepsilon}.
\end{equation*}
\end{theorem}

In \cite{BG2}, this result was extended to lines in higher dimensional spaces. In this paper, we prove the following analogue of Theorem
\ref{real} for lines in $\mathbb{C}^2$, where the exponent  $1/5$ is replaced by $1/12$. 

\begin{theorem} \label{complex} Let $\varepsilon>0$ and let $c\in \mathbb{C}\setminus\mathbb{Q}(i)$. Then for almost all $\alpha\in \mathbb{C}$, 
with respect to the Lebesgue measure, there are infinitely many triples $(p,q,r)$ with $p$ and $r$ Gaussian primes and $q$ a Gaussian integer such that
\begin{equation} \label{simultan}
|p\alpha - r| \le |p|^{-1/12+\varepsilon} \quad \mbox{and} \quad |pc\alpha - q| \le |p|^{-1/12+\varepsilon}.
\end{equation}
\end{theorem}

We note that an inhomogeneous analogue of Theorem \ref{complex} can be established by minor modifications of the arguments in this paper. For
simplicity, we here consider only the homogeneous case of a line passing through the origin. 

The structure of our proof resembles that of Theorem \ref{real}, but the technical details are more
involved. In particular, a slight extension of a new result by the author \cite{Bai} on Diophantine approximation of numbers in $\mathbb{C}\setminus \mathbb{Q}(i)$
by fractions of Gaussian integers with Gaussian prime denominator plays a significant role in this paper. 

We note that the work in \cite{BG1}, in which Theorem \ref{real} was established, was motivated by work of G. Harman and H. Jones \cite{HJ} 
on a similar problem about restricted Diophantine approximations to points on a curve. They proved the following result.

\begin{theorem} \label{HaJo} Let $\varepsilon>0$ and $\tau> 1$. Then for almost all positive $\alpha$ there are
infinitely many $p$, $q$, $r$, all prime, such that
\begin{equation*} 
0 < p\alpha - r \le p^{-1/6+\varepsilon} \quad \mbox{and} \quad 0< p\alpha^{\tau} - q \le p^{-1/6+\varepsilon}.
\end{equation*}
\end{theorem}

The complex analogue of Theorem \ref{HaJo} for $\tau\in \mathbb{N}\setminus \{1\}$, in particular the case $\tau=2$ of the complex parabola, would 
certainly be a very interesting problem to consider as well. \\   

{\bf Conventions.} 
(1) Throughout the sequel, we shall assume that $0<|c|\le 1$ in Theorem \ref{complex}. The case $|c|>1$ can be treated similarly, by minor modifications of the method.\\ 
(2) Throughout this paper, $\varepsilon$ is a small enough positive real number.\\

{\bf Acknowledgement.} The author would like to thank Prof. Anish Ghosh for useful discussions about this topic at and after a pleasant stay
at the Tata Institute in Mumbai in August 2016.

\section{A Metrical approach}
Our basic approach is an extension of that in \cite[section 2]{BG1} (see also \cite[section 2]{BG2}) and has its origin in \cite{HJ}. We
first establish the metrical lemma below. Our proof follows closesly the arguments in \cite[Proof of Lemma 1]{HJ}.
Throughout the sequel, we denote by $\mu(\mathcal{C})$ the Lebesgue measure of a measurable set $\mathcal{C}\subseteq \mathbb{C}$ and we write
$$
D(a,b):=\left\{z\in \mathbb{C}\ :\ a< |z|\le b\right\}
$$
and 
$$
D(a,b,\gamma_1,\gamma_2):=\left\{Re^{i\theta}\ :\ a< R\le b,\ \gamma_1< \theta\le \gamma_2\right\}.
$$ 

\begin{lemma} \label{metric}
Let $\mathcal{S}$ be a subset of the positive integers. Assume that $A$ and $B$ are reals such that $0<A<B$ and let $\mathcal{M}:=D(A,B)$.
Let $F_N(\alpha)$ be a non-negative real-valued function of $N$, an element of $\mathcal{S}$,
and $\alpha$, a complex number. Let further $G_N$ and $V_N$ be real-valued functions of $N\in \mathcal{S}$ such that the following hold. 

\begin{equation}
G_N \rightarrow \infty \quad \mbox{as } N\in \mathcal{S} \mbox{ and } N\rightarrow\infty.
\end{equation}

\begin{equation} \label{VNbound}
V_N=o\left(G_N\right) \quad \mbox{as } N\in \mathcal{S} \mbox{ and } N\rightarrow\infty. 
\end{equation}

\begin{equation} \label{FNint}
\begin{cases}
\mbox{For all } a,b,\gamma_1,\gamma_2 \mbox{ with } A\le a<b\le B \mbox{ and } -\pi<\gamma_1<\gamma_2\le \pi \mbox{ we have}\\
\limsup\limits\limits_{\substack{N\in \mathcal{S}\\ N\rightarrow \infty}} \int\limits_{\gamma_1}^{\gamma_2} \int\limits_a^b 
\frac{F_N\left(Re^{i\theta}\right)}{G_N} \ dR\ d\theta \ge (\gamma_2-\gamma_1)\left(b^2-a^2\right).
\end{cases}
\end{equation}

\begin{equation} \label{constantK}
\begin{cases}
\mbox{There is a positive constant } K \mbox{ such that, for any measurable set } \mathcal{C}\subseteq\mathcal{M} \\
\mbox{ and any } N\in \mathcal{S},\ 
\int\limits_{\mathcal{C}} F_N\left(Re^{i\theta}\right) \ dR\ d\theta\le KG_N\mu(\mathcal{C})+V_N.
\end{cases}
\end{equation}
Then for almost all $\alpha\in \mathcal{M}$, we have 
$$
\limsup\limits_{\substack{N\in \mathcal{S}\\ N\rightarrow\infty}} \frac{F_N(\alpha)}{G_N}\ge 1. 
$$
\end{lemma}

\begin{proof} We write
$$
H_N(\alpha):=\frac{F_N(\alpha)}{G_N}
$$
and suppose that
$$
\limsup\limits_{\substack{N\in \mathcal{S}\\ N\rightarrow\infty}} H_N(\alpha)<1
$$
on a subset of $\mathcal{M}$ with positive measure. Then there must be a set $\mathcal{A}\subset \mathcal{M}$ with positive measure and a 
constant $c<1$ with
\begin{equation} \label{limsupbound}
\limsup\limits_{\substack{N\in \mathcal{S}\\ N\rightarrow\infty}} H_N(\alpha)\le c \quad \mbox{for all } \alpha\in \mathcal{A}.
\end{equation}
By the Lebesgue density theorem, for each $\varepsilon>0$ there are $a,b,\gamma_1,\gamma_2$ with $A\le a<b\le B$, 
$-\pi\le \gamma_1<\gamma_2\le \pi$ and 
$(\gamma_2-\gamma_1)(b^2-a^2)<1$
such that, if we put
$\mathcal{B}:=\mathcal{A}\cap \mathcal{Z}$ with $\mathcal{Z}:= D(a,b,\gamma_1,\gamma_2)$,
then 
$$
\mu(\mathcal{B})>(1-\varepsilon)\mu(\mathcal{Z})=(1-\varepsilon)(\gamma_2-\gamma_1)(b^2-a^2)
$$
and hence
$$
\mu(\mathcal{Z}\setminus \mathcal{B})< \varepsilon(\gamma_2-\gamma_1)(b^2-a^2)<\varepsilon.
$$
Now, using \eqref{constantK},
\begin{equation*}
\begin{split}
\int\limits_{\mathcal{Z}} H_N(\alpha) \ dR\ d\theta =& 
\int\limits_{\mathcal{B}} H_N(\alpha) \ dR\ d\theta + \int\limits_{\mathcal{Z}\setminus \mathcal{B}} H_N(\alpha) \ dR\ d\theta\\
\le & \int\limits_{\mathcal{B}} H_N(\alpha) \ dR\ d\theta +K\varepsilon+\frac{V_N}{G_N},
\end{split}
\end{equation*}
where $\arg(\alpha)=\theta$ and $|\alpha|=R$. So if 
\begin{equation*}
\varepsilon:=\frac{(1-c)(\gamma_2-\gamma_1)\left(b^2-a^2\right)}{2K},
\end{equation*}
then, in view of \eqref{VNbound} and \eqref{limsupbound}, it follows that
\begin{equation*}
\limsup\limits_{\substack{N\in \mathcal{S}\\ N\rightarrow\infty}}  \int\limits_{\mathcal{Z}} H_N(\alpha) \ dR\ d\theta \le
c\mu(\mathcal{B})+K\varepsilon=c(\gamma_2-\gamma_1)\left(b^2-a^2\right)+K\varepsilon<(\gamma_2-\gamma_1)\left(b^2-a^2\right).
\end{equation*}
This contradicts \eqref{FNint} and so completes the proof. 
\end{proof}

Now let $F_N(\alpha)$ be the number of solutions to \eqref{simultan} with $|p|\le N$ and for $0<A<B$ let
$$
G_N(A,B):=  C\cdot \frac{A}{B} \cdot \frac{N^{5/3+4\varepsilon}}{\log^2 N}, 
$$
where $C>0$ is a suitable constant only depending on $c$. 
In the remainder of this paper, we will prove the following.

\begin{theorem} \label{Theo}
There exists $C=C(c)>0$ and an infinite set $\mathcal{S}$ of natural numbers $N$ such that the following hold.\\ 

(i) Let $0 < A < B$ be given. Then for all $a,b,\gamma_1,\gamma_2$ with $A\le a<b\le B$ and $-\pi<\gamma_1<\gamma_2\le \pi$ we have
\begin{equation*}
\int\limits_{\gamma_1}^{\gamma_2} \int\limits_a^b F_N\left(Re^{i\theta}\right) \ dR\ d\theta \ge (\gamma_2-\gamma_1)\left(b^2-a^2\right)G_N(A,B)
\end{equation*}
if $N\in \mathcal{S}$ and $N$ large enough.\\

(ii) Let $0<A<B$ be given. Then there exists a constant $K=K(A,B)$ such that, for every $\alpha\in \mathbb{C}$
with $A\le |\alpha|\le B$, we have 
$$
F_N(\alpha)\le KG_N(A,B)+J_N(\alpha)
$$
with
$$
\int\limits_{-\pi}^{\pi} \int\limits_A^B \left|J_N\left(Re^{i\theta}\right)\right| dR d\theta=o\left(G_N(A,B)\right) 
$$
if $N\in \mathcal{S}$ and $N\rightarrow \infty$. 
\end{theorem}

Together with Lemma \ref{metric}, this implies Theorem \ref{complex}.

\section{Results on Gaussian primes in sectors}
In this section, we provide two results which are related to the distribution of Gaussian primes which are of key importance in this paper.
By $\pi(P_1,P_2,\omega_1,\omega_2)$, we denote the number of primes in $D(P_1,P_2,\omega_1,\omega_2)$. The prime number theorem for
Gaussian primes in sectors due to Kubylius \cite{Kub} implies the following.

\begin{theorem}\label{PNT} If $0\le P_1<P_2$ and $\omega_1<\omega_2\le \omega_1+2\pi$, then 
$$
\pi(P_1,P_2,\omega_1,\omega_2) = \frac{(\omega_2-\omega_1)\left(P_2^2-P_1^2\right)+o\left(P_2^2\right)}{\log P_2^2}
$$
as $P_2\rightarrow \infty$. 
\end{theorem}

Further, for $\delta>0$ and $c\in \mathbb{C}$ we denote by  
$\pi(P_1,P_2,\omega_1,\omega_2;\delta,c)$ the number of  Gaussian primes $p$ contained in 
$D(P_1,P_2,\omega_1,\omega_2)$ such that 
$$
\min\limits_{q\in \mathbb{Z}[i]} |pc-q|\le \delta,
$$
and by $\pi^{\ast}(P_1,P_2,\omega_1,\omega_2;\delta,c)$ we denote the number of Gaussian primes $p$ contained in $D(P_1,P_2,\omega_1,\omega_2)$ such that
$$
\min\limits_{q\in \mathbb{Z}[i]} \max\left(|\Re(pc-q)|,|\Im(pc-q)|\right) \le \delta. 
$$
We have the following. 

\begin{theorem}\label{signi1} Let $c\in \mathbb{C}\setminus \mathbb{Q}(i)$. Then there exists an increasing sequence of natural numbers
$(M_k)_{k\in \mathbb{N}}$ such that the following holds. If $0\le P_1<P_2\le M_k$, $\omega_1<\omega_2\le \omega_1+2\pi$ and 
$M_k^{\varepsilon-1/12}<\delta_k\le 1/2$, then  
$$
\pi^{\ast}(P_1,P_2,\omega_1,\omega_2;\delta_k,c) = 
4\delta_k^2\pi^{\ast}(P_1,P_2,\omega_1,\omega_2)+o\left(\frac{\delta_k^2 M_k^2}{\log M_k}\right) 
$$
as $k\rightarrow \infty$. 
\end{theorem}

Since 
$$
\pi(P_1,P_2,\omega_1,\omega_2;\delta,c)\ge \pi^{\ast}(P_1,P_2,\omega_1,\omega_2;\delta/\sqrt{2},c),
$$
we immediately deduce the following from Theorem \ref{PNT} and Theorem \ref{signi1}.

\begin{corollary}\label{signi} Under the conditions of Theorem \ref{signi1}, we have 
$$
\pi(P_1,P_2,\omega_1,\omega_2;\delta_k,c) \ge \frac{\delta_k^2(\omega_2-\omega_1)\left(P_2^2-P_1^2\right)+o\left(\delta_k^2 M_k^2\right)}{\log M_k}
$$
as $k\rightarrow \infty$.  
\end{corollary}

{\it Proof of Theorem \ref{signi1}}. 
Theorem \ref{signi1} is a generalization of Theorem 10.1 in \cite{Bai}, in which the same result was proved for the special case when 
$P_2^2=M_k^2=N_k$, $P_1^2=M_k^2/2=N_k/2$, $\omega_1=-\pi$ and $\omega_2=\pi$. The proof of Theorem \ref{signi1} above goes along the same lines. We indicate 
in the following which alterations need to be made.

In section 3 in \cite{Bai}, the sets $A$ and $B$ need to be replaced by
$$
A:=\left\{n\in D(P_1,P_2,\omega_1,\omega_2;\delta,c)\ :\  \min\limits_{q\in \mathbb{Z}[i]}\max\left(|\Re(pc-q)|,|\Im(pc-q)|\right)\le\delta\right\}
$$
and 
$$
B:=D(P_1,P_2,\omega_1,\omega_2;\delta,c),
$$
$P_1^2$ taking the role of $x/2$ and $P_2^2$ taking that of $x$. The subsequent alterations in sections 3 to 6, where the relevant terms are
boilt down to linear exponential sums, are obvious. (In particular, the arguments of the products $mn$ need to be restricted to lie in
$(\omega_1,\omega_2] \bmod{2\pi}$.) Here the linear exponential sums in question take the form
$$
\sum\limits_{\substack{\tilde{y}<|m|\le y\\ \arg(m)\in (\tilde{\omega}_1,\tilde{\omega}_2] \bmod{2\pi}}} e(\Im(m\kappa))
$$
for some $\tilde{y},y\in \mathbb{R}$ with $0<\tilde{y}<y$, $\tilde{\omega}_1,\tilde{\omega}_2\in \mathbb{R}$ with $
\tilde{\omega}_1<\tilde{\omega}_2\le \tilde{\omega_1}+2\pi$ and $\kappa\in \mathbb{C}$. These are the same 
linear exponential sums as in section 7 in \cite{Bai} with the extra condition that $m$ lies in a sector. The same splitting argument
as in section 7 applies in this more general situation and leads to the same estimates. The rest of the proof is similar as before, where
$P_2^2$ takes the role of $x$ and $M_k^2$ takes the role of $N_k$. $\Box$\\

{\bf Remark}:
Following \cite{Bai}, an admissible choice for the $M_k$'s are the sixth powers of absolute values of the Hurwitz continued fraction approximants of
$c$. Throughout the sequel, we assume that this is the case.

\section{Proof of Theorem \ref{Theo}(i)}
We set
$$
\mathcal{S}:=\{M_1,M_2,...\}.
$$
and suppose that $N\in \mathcal{S}$. Further, we write
$$
B_{\Delta}\left(a\right):= \left\{x\in \mathbb{C}\ :\ |x-a|\le \Delta\right\}
$$
and denote by $\mathbb{G}$ be the set of Gaussian primes. 
Let 
$$
\mathcal{A}_p=\bigcup_{\substack{r\in \mathbb{G}\\ q\in \mathbb{Z}[i]}} B_{|\eta/p|}\left(
\frac{r}{p}\right) \cap B_{|\eta/(cp)|}\left(\frac{1}{c}\cdot \frac{q}{p}\right) \cap D(a,b,\gamma_1,\gamma_2),
$$
where $\eta:=|p|^{\varepsilon-1/12}$. Then
\begin{equation} \label{integral}
\int\limits_{\gamma_1}^{\gamma_2}\int\limits_{a}^{b} F_N\left(re^{i\theta}\right)\ dr\ d\theta =
\sum\limits_{\substack{p \in \mathbb{G}\\ |p|\le N}} \mu(\mathcal{A}_p).
\end{equation}

Set
\begin{equation} \label{Mdefi}
M:=(a+b)/(2a)
\end{equation}
and 
$$
L:=4\left[\frac{2\pi}{\gamma_2-\gamma_1}\right]. 
$$
Our strategy is to split the summation over $p$ on the right-hand side of \eqref{integral} into summations over sets of the form 
$D(P,MP,\omega,\omega+2\pi/L)$ with $MP\le N$ and derive lower bounds. Clearly,
\begin{equation} \label{AB}
\sum\limits_{p \in \mathbb{G} \cap D(P,MP,\omega,\omega+2\pi/L) } \mu(\mathcal{A}_p)\ge
\sum\limits_{p \in \mathbb{G} \cap D(P,MP,\omega,\omega+2\pi/L) } \mu(\mathcal{B}_p)
\end{equation}
with
$$
\mathcal{B}_p=\bigcup_{\substack{r\in \mathbb{G}\\ q\in \mathbb{Z}[i]}} B_{|\eta'/p|}\left(
\frac{r}{p}\right) \cap B_{|\eta'/(cp)|}\left(\frac{1}{c}\cdot \frac{q}{p}\right) \cap D(a,b,\gamma_1,\gamma_2),
$$
where 
\begin{equation} \label{etadef}
\eta':=(MP)^{\varepsilon-1/12}.
\end{equation}

We note that if $|p|\le MP$ and
$$
\frac{1}{c}\cdot \frac{q}{p} \in B_{|\eta'/p|}\left(\frac{r}{p}\right),
$$
which latter is equivalent to
$$
q \in B_{|\eta'c|}(rc),
$$
then 
$$
\mu\left(B_{|\eta'/p|}\left(
\frac{r}{p}\right) \cap B_{|\eta'/(cp)|}\left(\frac{1}{c}\cdot \frac{q}{p}\right)\right)\ge \nu,
$$
where
\begin{equation} \label{nudef}
\nu:=\left(\frac{\pi}{3}-\frac{\sqrt{3}}{2}\right)\cdot \left|\frac{\eta'}{MP}\right|^2=\left(\frac{\pi}{3}-\frac{\sqrt{3}}{2}\right)\cdot (MP)^{2\varepsilon-13/6}.
\end{equation}
Here we use our condition that $0<|c|\le 1$.
Also, for all $p\in D(P,MP,\omega,\omega+2\pi/L)$,
$$
r\in D(MPa,Pb,\gamma_1+\omega+2\pi/L,\gamma_2+\omega) \Longrightarrow \frac{r}{p}\in D(a,b,\gamma_1,\gamma_2).
$$
We thus have 
\begin{equation} \label{key}
\sum\limits_{p\in\D(P,MP,\omega,\omega+2\pi/L)} \mu(\mathcal{B}_p) \ge \nu N(P,\omega),
\end{equation}
where $N(P,\omega)$ counts the number of $(p,q,r)\in \mathbb{G} \times \mathbb{Z}[i]\times \mathbb{G}$ satisfying 
\begin{equation} \label{qprconds}
 p\in D(P,MP,\omega,\omega+2\pi/L), \quad q\in B_{\delta}(rc), \quad r\in D(MPa,Pb,\gamma_1+\omega+2\pi/L,\gamma_2+\omega),  
\end{equation}
where
\begin{equation} \label{deltadef}
\delta:=|\eta' c|=\frac{|c|}{(MP)^{1/12-\varepsilon}}.
\end{equation}
We note that
\begin{equation} \label{note1}
\frac{1}{4}\cdot \left(b^2-a^2\right)\cdot P^2 \le (Pb)^2-(MPa)^2=\frac{3}{4}\cdot \left(b^2-a^2\right) \cdot P^2
\end{equation}
and
\begin{equation} \label{note2}
\frac{\gamma_2-\gamma_1}{2}\le (\gamma_2+\omega)-(\gamma_1+\omega+2\pi/L)\le \gamma_2-\gamma_1.
\end{equation}

Using Theorem \ref{PNT}, the number $\pi(P,MP,\omega,\omega+2\pi/L)$ of Gaussian primes \\
$p\in D(P,MP,\omega,\omega+2\pi/L)$ is bounded from below by 
\begin{equation} \label{pcount}
\pi(P,MP,\omega,\omega+2\pi/L) \ge \frac{2\pi}{L}\cdot \frac{(M^2-1)P^2+o(L(MP)^2)}{2\log N}.
\end{equation}
The number of $(q,r)\in \mathbb{Z}[i]\times \mathbb{G}$ satisfying
$$
q\in B_{\delta}(rc), \quad r\in D(MPa,Pb,\gamma_1+\omega+2\pi/L,\gamma_2+\omega)
$$ 
equals $\pi(MPa,Pb,\gamma_1+\omega+2\pi/L,\gamma_2+\omega;\delta,c)$ and is, by Corollary \ref{signi}, bounded from below by
\begin{equation} \label{qrcount}
\begin{split}
& \pi(MPa,Pb,\gamma_1+\omega+2\pi/L,\gamma_2+\omega;\delta,c)\\ 
\ge & 
\frac{\delta^2\left((\gamma_2+\omega)-(\gamma_1+\omega+2\pi/L)\right)((MPa)^2-(Pb)^2)+o(\delta^2 N^2)}{\log N}.
\end{split}
\end{equation}

Combing \eqref{AB}, \eqref{nudef}, \eqref{key}, \eqref{deltadef}, \eqref{note1}, \eqref{note2}, \eqref{pcount} and \eqref{qrcount}, we obtain
\begin{equation} \label{near}
\begin{split}
& \sum\limits_{p \in \mathbb{G} \cap D(P,MP,\omega,\omega+2\pi/L) } \mu(\mathcal{A}_p)\\ \ge & 
C(MP)^{4\varepsilon-7/3}\cdot \frac{1}{L}\cdot \frac{(M^2-1)P^2+o(L(MP)^2)}{\log N}\cdot 
\frac{(\gamma_2-\gamma_1)(b^2-a^2)P^2+o(N^2)}{\log N}
\end{split}
\end{equation}
for some constant $C=C(c)>0$ if $N\in \mathcal{S}$ and $N\rightarrow \infty$.
By splitting the interval $[1,N)$ into intervals of the form $(P,MP]=(N/M^k,N/M^{k-1}]$ and summing up, it follows from \eqref{near} that
\begin{equation*}
\begin{split}
& \sum\limits_{p \in \mathbb{G} \cap D(1,N,\omega,\omega+2\pi/L)} \mu(\mathcal{A}_p)\\ \ge & 
\frac{C(M^2-1)(\gamma_2-\gamma_1)(b^2-a^2)N^{5/3+4\varepsilon}}{M^{7/3-4\varepsilon}L\log^2 N} \cdot \sum\limits_{k=1}^{\infty} M^{-(5/3+4\varepsilon)k} \cdot (1+o(1))\\
\ge & \frac{C(\gamma_2-\gamma_1)(b^2-a^2)N^{5/3+4\varepsilon}}{M^{2/3}L\log^2 N} 
\cdot (1+o(1))\\
\ge & C\cdot \frac{A}{B} \cdot \frac{(\gamma_2-\gamma_1)(b^2-a^2)N^{5/3+4\varepsilon}}{L\log^2 N}
\end{split}
\end{equation*}
if $N\in \mathcal{S}$ is large enough, where for the last line, we have used \eqref{Mdefi} and $A\le a<b\le B$.   
Splitting the interval $(0,2\pi]$ into intervals of the form $(\omega,\omega+2\pi/L]=(2\pi (k-1)/L,2\pi k/L]$ with $k=1,...,L$, it further follows that
$$
\sum\limits_{\substack{p \in \mathbb{G}\\ |p|\le N}} \mu(\mathcal{A}_p) = 
\sum\limits_{k=1}^L \sum\limits_{p \in \mathbb{G} \cap D(1,N,2\pi (k-1)/L,2\pi k/L)} \mu(\mathcal{A}_p) \ge
C\cdot \frac{A}{B}  \cdot \frac{(\gamma_2-\gamma_1)(b^2-a^2)N^{5/3+4\varepsilon}}{\log^2 N}
$$
if $N$ is large enough. Combining this with \eqref{integral} completes the proof of Theorem \ref{Theo}(i).  

\section{Proof of Theorem \ref{Theo}(ii)} 

\subsection{Sieve theoretical approach}
We extend the treatment in \cite[section 5]{BG1} (see also \cite[section 4]{BG2}), which has its origin in \cite{HJ}, to the situation in $\mathbb{Z}[i]$.  
We point out that there is a mistake in \cite[section 5]{BG1}: The set $\mathcal{A}$ should consist of products of the form
$n[n\alpha]$, not of the form $n[n\alpha][nc\alpha]$, and 
we bound the number of $n$'s such that $n[n\alpha]$ is the product of two primes, not three primes. This mistake, however, doesn't 
affect the method and the final result. As in \cite{Bai}, we define 
$$
||z||:=\max\left\{||\Re(z)||,||\Im(z)||\right\},
$$
where $\Re(z)$ is the real part and $\Im(z)$ is the imaginary part of $z\in \mathbb{C}$, and for $x\in \mathbb{R}$, $||x||$ denotes the distance of $x$ to the nearest integer. We further define
$$
f(z):=\tilde{f}(\Re(z))+\tilde{f}(\Im(z))i
$$
if $z\in \mathbb{C}$ and $||z||<1/2$, where $\tilde{f}(x)$ is the integer nearest to $x\not\in \mathbb{Z}+1/2$. 

First, we split the interval $(0,N]$ into dyadic intervals $(P/2,P]$ with $P:=N/2^k$, $k=0,1,2,...$. Then we write 
$$
\mathcal{A}_P(\alpha)=\left\{n\cdot f(n\alpha) \ :\  n\in \mathbb{Z}[i], P/2< |n|\le P, \ \max\{||n\alpha||,||nc\alpha||\}\le
\mu\right\},
$$
where
\begin{equation} \label{muedef}
\mu:=\left(\frac{P}{2}\right)^{\varepsilon-1/12}
\end{equation}
if $P^{\varepsilon-1/12}<1/2$, i.e.
\begin{equation} \label{Pcondit}
P>2^{1+1/(1/12-\varepsilon)}.
\end{equation}
It follows that
\begin{equation} \label{upperFN}
F_N(\alpha)\le \sum\limits_{\substack{0\le k\le 1+\log_2\left(N/2^{1/(1/12-\varepsilon)}\right)}}
\sharp \left(\mathbb{G}_2\cap \mathcal{A}_{N/2^k}(\alpha)\right)+O(1),
\end{equation}
where $\mathbb{G}_2$ is the set of products of two Gaussian primes. We bound 
$\sharp\left(\mathbb{G}_2\cap \mathcal{A}_{N/2^k}(\alpha)\right)$ from above using a simple two-dimensional upper bound 
sieve in the setting of Gaussian integers, which is obtained by a standard application of the Selberg sieve in the setting of Gaussian
integers. \\

\begin{lemma} \label{uppersieve} Let $\mathcal{N}$ be a subset of the Gaussian integers and $f_1,f_2:\mathcal{N} \rightarrow \mathbb{Z}[i]$ two functions. 
For $P\ge 2$ and $d_1,d_2\in \mathbb{Z}[i]\setminus\{0\}$ let $S_P(d_1,d_2)$ be the number of $n\in \mathcal{N}$ such that
\begin{equation} \label{satisfy}
f_1(n)\equiv 0 \bmod{d_1}, \quad f_2(n)\equiv 0 \bmod{d_2}, \quad P/2< |n|\le P
\end{equation}
and $G_P(d_1,d_2)$ the number of $n\in \mathcal{N}$ satisfying \eqref{satisfy} such that $f_1(n)$ and $f_2(n)$ are both Gaussian primes. Then
for any $X>0$ and $\varepsilon>0$,
$$
G_P(d_1,d_2)\le \frac{C(\varepsilon)XP^2}{(\log P)^2}+ 
O\left(\sum\limits_{\substack{d_1,d_2\in \mathbb{Z}[i]\setminus \{0\}\\ 1\le |d_1|,|d_2|\le P^{\varepsilon}}} 
|d_1d_2|^{\varepsilon}\left|S_P(d_1,d_2) - \frac{XP^2}{|d_1|^2|d_2|^2}\right|\right)
$$
as $P\rightarrow \infty$, where $C(\varepsilon)$ is a constant depending only on $\varepsilon$. 
\end{lemma}

Here we consider the case when
\begin{equation} \label{max}
\mathcal{N}:=\{n\in \mathbb{Z}[i] : \max(||n\alpha||,||nc\alpha||)\le \mu\},
\end{equation}
$f_1(n)=n$ and $f_2(n)=f(n\alpha)$. We write $S_P(\alpha;d_1,d_2):=S_P(d_1,d_2)$. Clearly, $S_P(\alpha;d_1,d_2)$ equals the number of
$n\in \mathbb{Z}[i]$ with $P/2<|n|\le P$ such that
\begin{equation} \label{threeconds}
\frac{P}{2|d_1|}<|n|\le \frac{P}{|d_1|}, \quad \left|\left|\frac{nd_1\alpha}{d_2}\right|\right|\le \frac{\mu}{|d_2|}, \quad 
||nd_1c\alpha||\le \mu.
\end{equation}
Heuristically, $S_P(\alpha;d_1,d_2)$ should behave like $12\pi P^2\mu^4/(|d_1|^2|d_2|^2)$. Therefore, we write 
\begin{equation} \label{write}
S_P(\alpha;d_1,d_2)=\frac{12\pi P^2\mu^4}{|d_1|^2|d_2|^2}+ E_P(\alpha;d_1,d_2).
\end{equation}
Then, applying Lemma \ref{uppersieve} gives
\begin{equation}
\sharp\left(\mathbb{G}_2\cap \mathcal{A}_P\right)\le \frac{C(\varepsilon)P^2\mu^4}{\log^2 P} + O\left(\tilde{J}_P(\alpha)\right),
\end{equation}
where 
\begin{equation*}
\begin{split}
\tilde{J}_P(\alpha):= \sum\limits_{1\le |d_1d_2|\le P^{\varepsilon}} |d_1d_2|^{\varepsilon}|E_P(\alpha;d_1,d_2)|.
\end{split}
\end{equation*}
Hence, by \eqref{upperFN}, to establish the claim in Theorem \ref{Theo}(ii), it suffices to show that
\begin{equation} \label{average}
\begin{split}
& \sum\limits_{\substack{0\le k\le 1+\log_2(N/2^{1/(1/12-\varepsilon)})}}\sum\limits_{1\le |d_1d_2|\le N^{\varepsilon}} |d_1d_2|^{\varepsilon}\int\limits_{-\pi}^{\pi} \int\limits_A^B 
\left|E_{N/2^k}(Re^{i\theta};d_1,d_2)\right|\ dR \ d\theta\\  = & 
o\left(\frac{N^2\mu^4}{\log^2 N}\right) 
\end{split}
\end{equation}
as $N\rightarrow \infty$ and $N\in \mathcal{S}$.

\subsection{Fourier analysis}
Throughout the sequel, we assume that \eqref{Pcondit} is satisfied. 
We use Fourier analysis to express $E_P(\alpha;d_1,d_2)$ in terms of trigonometrical polynomials. We have

\begin{equation*}
\begin{split}
S_P(\alpha;d_1,d_2)= \sum\limits_{\substack{n\in \mathbb{Z}[i]\\ P/|2d_1|<|n|\le P/|d_1|}} &
\left(\left[\Re\left(\frac{nd_1\alpha}{d_2}\right)+\frac{\mu}{|d_2|}\right]-\left[\Re\left(\frac{nd_1\alpha}{d_2}\right)-\frac{\mu}{|d_2|}
\right]
\right)\times\\ &
\left(\left[\Im\left(\frac{nd_1\alpha}{d_2}\right)+\frac{\mu}{|d_2|}\right]-\left[\Im\left(\frac{nd_1\alpha}{d_2}\right)-\frac{\mu}{|d_2|}\right]
\right)\times\\ & 
\left(\left[\Re\left(nd_1c\alpha\right)+\mu\right]-\left[\Re\left(nd_1c\alpha\right)-\mu\right]\right)\times\\ &
\left(\left[\Im\left(nd_1c\alpha\right)+\mu\right]-\left[\Im\left(nd_1c\alpha\right)-\mu\right]\right),
\end{split}
\end{equation*}
where $[x]$ is the integral part of $x\in \mathbb{R}$. 
Writing $\psi(x)=x-[x]-1/2$, it follows that
\begin{equation*}
\begin{split}
& S_P(\alpha;d_1,d_2)= \\ & \sum\limits_{\substack{n\in \mathbb{Z}[i]\\ P/|2d_1|<|n|\le P/|d_1|}} 
\left(\psi\left(\Re\left(\frac{nd_1\alpha}{d_2}\right)-\frac{\mu}{|d_2|}\right)-
\psi\left(\Re\left(\frac{nd_1\alpha}{d_2}\right)+\frac{\mu}{|d_2|}\right)
+\frac{2\mu}{|d_2|}\right)\times\\ &
\left(\psi\left(\Im\left(\frac{nd_1\alpha}{d_2}\right)-\frac{\mu}{|d_2|}\right)-
\psi\left(\Im\left(\frac{nd_1\alpha}{d_2}\right)+\frac{\mu}{|d_2|}\right)+\frac{2\mu}{|d_2|}\right)\times\\ & 
\left(\psi\left(\Re\left(nd_1c\alpha\right)-\mu\right)-\psi\left(\Re\left(nd_1c\alpha\right)+\mu\right)+2\mu\right)\times\\ &
\left(\psi\left(\Im\left(nd_1c\alpha\right)-\mu\right)-\psi\left(\Im\left(nd_1c\alpha\right)+\mu\right)+2\mu\right).
\end{split}
\end{equation*}
Next, we approximate the function $\psi(x)$ by a trigonomtrical polynomial using the following lemma due to Vaaler. 

\begin{lemma}[Vaaler] \label{Vaaler} For $0<|t|<1$ let
$$
W(t)=\pi t(1-|t|) \cot \pi t +|t|.
$$
Fix a natural number $J$. For $x\in \mathbb{R}$ define 
$$
\psi^{\ast}(x):=-\sum\limits_{1\le |j|\le J} (2\pi i j)^{-1}W\left(\frac{j}{J+1}\right)e(jx)
$$
and
$$
\sigma(x):=\frac{1}{2J+2} \sum\limits_{|j|\le J} \left(1-\frac{|j|}{J+1}\right)e(jx).
$$
Then $\sigma(x)$ is non-negative, and we have 
$$
|\psi^{\ast}(x)-\psi(x)|\le \sigma(x)
$$
for all real numbers $x$. 
\end{lemma}

\begin{proof} This is \cite{GKo}, Theorem A6.
\end{proof}

Let 
 \begin{equation} \label{J1J2def}
 J_1:= \left[\frac{N^{\varepsilon}|d_2|}{\mu}\right]  \quad \mbox{and} \quad  J_2:=\left[\frac{N^{\varepsilon}}{\mu}\right].
 \end{equation}
 Then from Lemma \ref{Vaaler}, using
$$
\frac{1}{j}\cdot (e(jx)-e(-jx))\ll x
$$
for any $j\in \mathbb{N}$ and $x\in \mathbb{R}$, we deduce that
\begin{equation} \label{deduce}
\begin{split}
S_P(\alpha;d_1,d_2)= & \frac{16\mu^4}{|d_2|^2}\cdot \sum\limits_{\substack{n\in \mathbb{Z}[i]\\ P/|2d_1|<|n|\le P/|d_1|}} 1 + O\left(1+\frac{P^2}{J_1^2J_2^2|d_1|^2}+F_P(\alpha;d_1,d_2)\right)\\
= & \frac{12\pi P^2\mu^4}{|d_1|^2 |d_2|^2} + O\left(\frac{P\mu^2}{|d_1| |d_2|}+\frac{P^2}{J_1^2J_2^2|d_1|^2}+F_P(\alpha;d_1,d_2)\right), 
\end{split}
\end{equation}
where 
\begin{equation*}
\begin{split}
F_P(\alpha;d_1,d_2)= & \frac{\mu^4}{|d_2|^2}\cdot \sum\limits_{\substack{(m_1,m_2,m_3,m_4)\in \mathbb{Z}^4\setminus\{(0,0,0,0)\}
\\ |m_1|\le J_1,\ |m_2|\le J_1\\ |m_3|\le J_2,\ |m_4|\le J_2}} 
\\ & \Big| \sum\limits_{\substack{n\in \mathbb{Z}[i]\\ P/|2d_1|<|n|\le P/|d_1|}} e\Big(m_1\cdot \Re\Big(\frac{nd_1\alpha}{d_2}\Big)+ m_2\cdot 
\Im\Big(\frac{nd_1\alpha}{d_2}\Big)\\ & +m_3\cdot \Re(nd_1c\alpha)+ m_4\cdot 
\Im(nd_1c\alpha)\Big) \Big|\\
= & \frac{\mu^4}{|d_2|^2}\cdot 
\sum\limits_{\substack{(m_1,m_2,m_3,m_4)\in \mathbb{Z}^4\setminus\{(0,0,0,0)\}\\ |m_1|\le J_1,\ |m_2|\le J_1\\ |m_3|\le J_2,\ |m_4|\le J_2}} 
\\ & \Big| \sum\limits_{\substack{n\in \mathbb{Z}[i]\\ P/|2d_1|<|n|\le P/|d_1|}} e\Big(\Im\Big(nd_1\alpha \cdot 
\Big(\frac{m_1+im_2}{d_2}+(m_3+im_4)c\Big)\Big)\Big) \Big|.
\end{split}
\end{equation*}
For the last line of \eqref{deduce}, we have used the elementary bound for the error term in the Gauss circle problem, namely
$$
\sum\limits_{\substack{n\in \mathbb{Z}[i]\\ |n|\le x}} 1 = \pi x^2+O(x).
$$
In the following, we will prove that
\begin{equation}  \label{end}
\int\limits_{-\pi}^{\pi} \int\limits_A^B 
\left|F_{P}(Re^{i\theta};d_1,d_2)\right|\ dR \ d\theta \ll \frac{N^{2-4\varepsilon}\mu^4}{|d_1|^2|d_2|^2}.
\end{equation} 
In view of \eqref{muedef}, \eqref{write}, \eqref{J1J2def} and \eqref{deduce},  this suffices to prove \eqref{average} and therefore establishes the claim of 
Theorem \ref{Theo}(ii). 

We first bound $F_P(\alpha;d_1,d_2)$ for individual $\alpha$, using the following result from \cite{Bai}. 

\begin{lemma} \label{small} Let $\kappa\in \mathbb{C}$ and $0<\tilde{x}<x$. Then 
\begin{equation} \label{lin}
\sum\limits_{\substack{n\in \mathbb{Z}[i]\\ \tilde{x}<|n|\le x}} e\left(\Im(n\kappa)\right) \ll x \cdot 
\min\left\{||\Im(\kappa)||^{-1},x\right\}^{1/2} \cdot \min\left\{||\Re(\kappa)||^{-1},x\right\}^{1/2}.
\end{equation}
\end{lemma}

\begin{proof} This is (31) in \cite{Bai}.
\end{proof}

It follows that
\begin{equation} \label{follows}
\begin{split}
& F_P(\alpha;d_1,d_2)\\ \ll &  
\frac{P\mu^4}{|d_1|\cdot |d_2|^2}\cdot  
\sum\limits_{\substack{(m_1,m_2,m_3,m_4)\in \mathbb{Z}^4\setminus\{(0,0,0,0)\}\\ |m_1|\le J_1,\ |m_2|\le J_1\\ |m_3|\le J_2,\ |m_4|\le J_2}}
\\ & \min\left\{\left|\left|\Im\left(d_1\Big(\frac{m_1+im_2}{d_2}+(m_3+im_4)c\Big)\alpha\right)\right|\right|^{-1}, \frac{P}{|d_1|}\right\}^{1/2} \times\\
& \min\left\{\left|\left|\Re\left(d_1\Big(\frac{m_1+im_2}{d_2}+(m_3+im_4)c\Big)\alpha\right)\right|\right|^{-1}, \frac{P}{|d_1|}\right\}^{1/2}\\
\le & \frac{P\mu^4}{|d_1|\cdot |d_2|^2}\cdot
\sum\limits_{\substack{(n_1,n_2)\in \mathbb{Z}[i]^2\setminus\{(0,0)\}\\ |n_1|\le 2J_1\\ |n_2|\le 2J_2}}
\min\left\{\left|\left|\Im\left(d_1\Big(\frac{n_1}{d_2}+n_2c\Big)\alpha\right)\right|\right|^{-1}, \frac{P}{|d_1|}\right\}^{1/2} \times\\
& \min\left\{\left|\left|\Re\left(d_1\Big(\frac{n_1}{d_2}+n_2c\Big)\alpha\right)\right|\right|^{-1}, \frac{P}{|d_1|}\right\}^{1/2}.\\
\end{split}
\end{equation}

\subsection{Average estimation for $F_P(\alpha;d_1,d_2)$} 
To bound the double integral on the left-hand side of \eqref{end}, 
we now use the following lemma. 

\begin{lemma}\label{av} Let $z\in \mathbb{C}$ and $Y>0$. Then
\begin{equation*}
\begin{split}
& \int\limits_{-\pi}^{\pi} \int\limits_{A}^{B} \min\left\{\left|\left|\Im\left(zRe^{\theta i}\right)\right|\right|^{-1}, Y\right\}^{1/2} \cdot
\min\left\{\left|\left|\Re\left(zRe^{\theta i}\right)\right|\right|^{-1}, Y\right\}^{1/2} \ dR \ d\theta\\
\ll_{A,B}  & \max\left\{1,|z|^{-1}\right\}\log (2+Y).
\end{split}
\end{equation*}
\end{lemma}

\begin{proof} By change of variables, we have
\begin{equation*}
\begin{split}
& \int\limits_{-\pi}^{\pi} \int\limits_{A}^{B} \min\left\{\left|\left|\Im\left(zRe^{\theta i}\right)\right|\right|^{-1}, Y\right\}^{1/2} \cdot
\min\left\{\left|\left|\Re\left(zRe^{\theta i}\right)\right|\right|^{-1}, Y\right\}^{1/2} \ dR \ d\theta\\ = &  
\frac{1}{|z|^2}\cdot \int\limits_{-\pi}^{\pi} \int\limits_{|z|A}^{|z|B} \min\left\{\left|\left|\Im\left(r e^{\theta i}\right)\right|\right|^{-1}, Y\right\}^{1/2} \cdot
\min\left\{\left|\left|\Re\left(re^{\theta i}\right)\right|\right|^{-1}, Y\right\}^{1/2} \ dr \ d\theta.
\end{split}
\end{equation*}
Changing from polar to affine coordinates, and using Cauchy-Schwarz, we get 
\begin{equation*}
\begin{split}
&  \frac{1}{|z|^2} \cdot \int\limits_{-\pi}^{\pi} \int\limits_{|z|A}^{|z|B} \min\left\{\left|\left|\Im\left(re^{\theta i}\right)\right|\right|^{-1}, Y\right\}^{1/2} \cdot
\min\left\{\left|\left|\Re\left(re^{\theta i}\right)\right|\right|^{-1}, Y\right\}^{1/2} \ dr \ d\theta\\
\le & \frac{1}{|z|^2}\cdot \int\limits_{|z|A/2}^{|z|B} \int\limits_{|z|A/2}^{|z|B}\min\left\{\left|\left|\Im\left(x+yi\right)\right|\right|^{-1}, Y\right\}^{1/2} \cdot
\min\left\{\left|\left|\Re\left(x+yi\right)\right|\right|^{-1}, Y\right\}^{1/2} \ dy \ dx\\
= & \frac{1}{|z|^2}\cdot \int\limits_{|z|A/2}^{|z|B}  \int\limits_{|z|A/2}^{|z|B}\min\left\{\left|\left|\Im\left(x+yi\right)\right|\right|^{-1}, Y\right\}^{1/2} \cdot
\min\left\{\left|\left|\Re\left(x+yi\right)\right|\right|^{-1}, Y\right\}^{1/2} \ dy \ dx\\
= & \frac{1}{|z|^2}\cdot \left( \int\limits_{|z|A/2}^{|z|B}  \min\left\{\left|\left| x\right|\right|^{-1}, Y\right\}^{1/2} \ dx\right)^2\\
\ll & \frac{1}{|z|}\cdot  \left(B-\frac{A}{2}\right)\cdot \int\limits_{|z|A/2}^{|z|B}  \min\left\{\left|\left| x\right|\right|^{-1}, Y\right\} \ dx\\
= &  \left(B-\frac{A}{2}\right)\cdot \int\limits_{A/2}^{B}  \min\left\{\left|\left| |z|x\right|\right|^{-1}, Y\right\} \ dx.
\end{split}
\end{equation*}
From Lemma 5.1 in \cite{Bai}, it follows that
$$
\int\limits_{A/2}^{B}  \min\left\{\left|\left| |z|x\right|\right|^{-1}, Y\right\} \ dx 
\ll_{A,B} \max\left\{1, \left|z\right|^{-1}\right\}\log (2+Y).
$$
Putting everything together proves the claim.
\end{proof}

From \eqref{follows} and Lemma \ref{av}, we deduce that
\begin{equation} \label{EP}
\begin{split}
& \int\limits_{-\pi}^{\pi} \int\limits_A^B 
\left|F_{P}(Re^{i\theta};d_1,d_2)\right|\ dR \ d\theta\\ \ll_{A,B} & \frac{P\mu^4\log P}{|d_1|\cdot |d_2|^2}\cdot 
\sum\limits_{\substack{(n_1,n_2)\in \mathbb{Z}[i]^2\setminus\{(0,0)\}\\ |n_1|\le 2J_1\\ |n_2|\le 2J_2}}
  \max\left\{1, \left|d_1\left(\frac{n_1}{d_2}+n_2c\right)\right|^{-1}\right\}.
\end{split}
\end{equation}

\subsection{Final estimation}
Clearly, if $0<|d_1|,|d_2|\le P^{\varepsilon}$ and $J_1\ge |d_2|/\mu$, then
\begin{equation} \label{EP1}
\begin{split}
& \sum\limits_{\substack{(n_1,n_2)\in \mathbb{Z}[i]^2\setminus\{(0,0)\}\\ |n_1|\le 2J_1\\ |n_2|\le 2J_2}} \max\left\{1, \left|d_1\left(\frac{n_1}{d_2}+n_2c\right)\right|^{-1}\right\}\\
\ll & \frac{1}{|d_1|}\cdot \sum\limits_{\substack{n_2\in \mathbb{Z}[i]\setminus\{0\}\\ |n_2|\le 2J_2}} \left(\min\limits_{n_1\in \mathbb{Z}[i]} \left|\frac{n_1}{d_2}+n_2c\right|\right)^{-1} +
J_2^2\cdot \sum\limits_{\substack{n\in \mathbb{Z}[i]\setminus\{0\}\\ |n|\le 3J_1}} \max\left\{1, \left|\frac{d_2}{d_1n}\right|\right\}\\
\ll & \frac{1}{|d_1|}\cdot \sum\limits_{\substack{n_2\in \mathbb{Z}[i]\setminus\{0\}\\ |n_2|\le 2J_2}} \left(\min\limits_{n_1\in \mathbb{Z}[i]} \left|\frac{n_1}{d_2}+n_2c\right|\right)^{-1} +
J_1^2J_2^2\
\end{split}
\end{equation}
Since $N$ is the sixth power of absolute value of a denominator of the Hurwitz continued fraction approximation of $c$, we have 
$$
c=\frac{a}{q}+O\left(\frac{1}{|q|^2}\right)
$$
for some $a,q\in \mathbb{Z}[i]$ with $|q|^6=N$. Hence,
$$
\left|\frac{n_1}{d_2}+n_2c\right| \ge \frac{1}{|d_2q|}+O\left(\frac{J_2}{|q|^2}\right).
$$
Since
$$
\frac{1}{|d_2q|}\ge P^{-\varepsilon}N^{-1/6} \ge N^{-1/6-\varepsilon},
$$
it follows that 
$$
\left|\frac{n_1}{d_2}+n_2c\right| \gg N^{-1/6-\varepsilon}
$$
if $n_1\in \mathbb{Z}[i]$, $n_2\in  \mathbb{Z}[i]\setminus\{0\}$, $|n_2|\le 2J_2$ and $N$ is large enough, 
where we recall that $J_2=[N^{\varepsilon}/\mu]\le N^{1/12+\varepsilon}$.  Hence, from \eqref{EP1}, we deduce that
\begin{equation} \label{EP2}
\sum\limits_{\substack{(n_1,n_2)\in \mathbb{Z}[i]^2\setminus\{(0,0)\}\\ |n_1|\le 2J_1\\ |n_2|\le 2J_2}} \max\left\{1, \left|d_1\left(\frac{n_1}{d_2}+n_2c\right)\right|^{-1}\right\}
\ll N^{1/6+\varepsilon}J_2^2 + J_1^2J_2^2.
\end{equation}

Combining \eqref{EP} and \eqref{EP2} , we obtain
\begin{equation*} 
\int\limits_{-\pi}^{\pi} \int\limits_A^B 
\left|F_{P}(Re^{i\theta};d_1,d_2)\right|\ dR \ d\theta\\ \ll_{A,B}  
\frac{P\mu^4\log P}{|d_1|\cdot |d_2|^2}\cdot \left(N^{1/6+\varepsilon}J_2^2 + J_1^2J_2^2\right).
\end{equation*}
from which \eqref{end} follows using \eqref{muedef} and \eqref{J1J2def}. This completes the proof of Theorem \ref{Theo}(ii).

\end{document}